\documentclass[10 pt]{amsart}

\usepackage{amssymb,amsmath,amsfonts,mathrsfs}
\usepackage[title]{appendix}

\newtheorem{theorem}{Theorem}[section]
\newtheorem{lemma}[theorem]{Lemma}
\newtheorem{cor}[theorem]{Corollary}
\newtheorem{prop}[theorem]{Proposition}

\theoremstyle{definition}
\newtheorem{definition}[theorem]{Definition}

\theoremstyle{remark}
\newtheorem{remark}[theorem]{Remark}

\numberwithin{equation}{section}

\newcommand{\dom}{\operatorname{Dom}}

\newcommand\CC{\mathbb{C}}
\newcommand\RR{\mathbb{R}}
\newcommand\ZZ{\mathbb{Z}}
\newcommand\NN{\mathbb{N}}
\newcommand\TT{\mathbb{T}}
\newcommand\del{\Delta}
\newcommand\p{\partial}

\newcommand\sml{S^{m}(\ZZ^n\times \TT^n)}

\newcommand\sob{H}

\begin{document}

\title[]{The essential adjointness of pseudo-differential operators on $\mathbb{Z}^n$}
\author{Ognjen Milatovic}
\address{Department of Mathematics and Statistics\\
         University of North Florida   \\
       Jacksonville, FL 32224 \\
        USA
           }
\email{omilatov@unf.edu}

\subjclass[2010]{47G30, 35S05}

\keywords{discrete pseudo-differential operator, discrete Sobolev space, elliptic operator, essentially adjoint pair, self-adjointness}

\begin{abstract}
In the setting of the lattice $\ZZ^n$ we consider a pseudo-differential operator $A$ whose symbol belongs to a class defined on $\ZZ^n\times \TT^n$, where $\TT^n$ is the $n$-torus. We realize $A$ as an operator acting between the discrete Sobolev spaces $H^{s_j}(\mathbb{Z}^n)$, $s_j\in\RR$, $j=1,2$, with the discrete Schwartz space serving as the domain of $A$. We provide a sufficient condition for the essential adjointness of the  pair $(A,\,A^{\dagger})$, where $A^{\dagger}$ is the formal adjoint of $A$.
\end{abstract}

\maketitle

\section{Introduction}\label{S:intro}
Difference equations and the corresponding pseudo-differential operators on the lattice $\ZZ^n$ play an important role in the discretization of continuous problems. In the last fifteen years, a number of researchers have studied various questions concerning these operators; see, for instance, the papers~\cite{CDK-20, CK-19, DW-13, gjbnm-16, Mol-10,Rab-10, Rab-4, Rab-9, Rod-11}. In parallel with these developments, based on the quantization described in~\cite{RT-10}, the authors of~\cite{Kalleji-15,Mol-09,Mol-11,Pirh-11} have investigated operators on the circle $\mathbb{S}^1$ and, more generally, on the $n$-torus $\TT^n$. Furthermore, the authors of~\cite{Cat-18,Cat-20,km-20,km-21-cot} have studied various properties of (generalized) pseudo-differential operators (the so-called $\ZZ$-operators and $S$-operators) related to a finite
measure space $S$ such that $L^2(S)$ is a separable Hilbert space.

A few years ago the authors of~\cite{BRK-20} developed a global symbol calculus for pseudo-differential operators on $\ZZ^n$. In this calculus, $\mathbb{Z}^n\times\TT^n$  plays the role of the phase space, whereby the frequency component belongs to $\TT^n$. As the frequency space $\TT^n$ is compact, there will be no improvement with respect to the decay of the frequency variable if one tries to construct the symbol calculus by mimicking the definition of the standard H\"ormander class on $\RR^n$. In the article~\cite{BRK-20}, the authors overcame this difficulty by using a similar strategy as in~\cite{RT-10} but  with switched roles of the space and frequency variables. It turns out that the symbol class $S^{m}_{\rho,\delta}(\ZZ^n\times\TT^n)$, with $m,\rho,\delta\in\RR$, introduced in~\cite{BRK-20} (see definition~\ref{D-1} below for the special case $\rho=1$ and $\delta=0$), bears some resemblance to that of the so-called $SG$-operators discussed in~\cite{DW-08}. Having developed the appropriate symbol calculus, the authors of~\cite{BRK-20} were able to provide the conditions for the $\ell^2(\ZZ^n)$-boundedness (here, $\ell^2(\ZZ^n)$ denotes the space of square summable complex-valued functions on $\ZZ^n$), compactness, and $H^{s}(\ZZ^n)$-boundedness of the corresponding operators (here, $H^s(\ZZ^n)$ indicates a discrete Sobolev space; see section~\ref{SS:s-2-6} below). Additionally, the paper~\cite{BRK-20} contains, among other things, the proof of weighted apriori estimates for difference equations, the proof of (sharp) G\r{a}rding inequality, and a discussion of the problem of unique solvability of parabolic equations on $\ZZ^n$.

Building on the paper~\cite{BRK-20}, the authors of~\cite{DK-20} studied, among other questions, the relationship between the maximal and minimal realizations in $\ell^2(\ZZ^n)$ of a pseudo-differential operator $A$ whose domain $\dom(A)$ is the Schwartz space $\mathcal{S}(\ZZ^n)$ (see section~\ref{SS:s-2-3} below for the definition of $\mathcal{S}(\ZZ^n)$). Here, in analogy with the definitions in chapter 13 of the book~\cite{W-pdo-book} (in the context of $\RR^n$), the ``minimal operator" $A_{\min}$ refers to the closure of $A|_{\mathcal{S}(\ZZ^n)}$ in $\ell^2(\ZZ^n)$, while the ``maximal operator" $A_{\max}$ is defined as follows: Let $f$ and $u$ be two functions in $\ell^2(\ZZ^n)$. One says that $u\in\dom(A_{\max})$ and $A_{\max}u=f$ if
\begin{equation*}
(u,A^{\dagger}v)=(f,v),
\end{equation*}
for all $v\in \mathcal{S}(\ZZ^n)$, where $(\cdot,\cdot)$ is the inner product in $\ell^2(\ZZ^n)$ and $A^{\dagger}$ is the formal adjoint of $A$ (see~(\ref{E:inner-l-2}) and~(\ref{E:dual-1}) below for the descriptions of $(\cdot,\cdot)$ and $A^{\dagger}$). After demonstrating various  properties of the Sobolev spaces $H^s(\ZZ^n)$ and establishing the appropriate analogue of Agmon–-Douglis–-Nirenberg inequalities, the authors of~\cite{DK-20} proved (in their theorem 3.19) that if $A$ is an elliptic operator whose symbol belongs to the class $S^{m}_{1,0}(\ZZ^n\times\TT^n)$ with $m>0$, then $A_{\min}=A_{\max}$. Subsequent to the paper~\cite{DK-20}, the authors of~\cite{km-21} considered weighted $M^{m}_{\rho,\Lambda}(\ZZ^n\times\mathbb{T}^n)$-type symbols (analogous to the class $M^{m}_{\rho,\Lambda}(\RR^n\times\RR^n)$ in ~\cite{GM-03,GM-05,W-06}), and after developing the corresponding calculus, established, among other things, the coincidence of minimal and maximal operators on $\ell^2(\ZZ^n)$  (assuming $M$-ellipticity of the corresponding symbols).

The question about the equality of maximal and minimal realizations of a pseudo-differential operator $A$ whose symbol belongs to $S^{m}_{1,0}(\ZZ^n\times\TT^n)$, $m\in\RR$, and whose domain is $\dom(A)=\mathcal{S}(\ZZ^n)$ can be rephrased as follows: under what conditions do $A$ and its formal adjoint $A^{\dagger}$ form an ``essentially adjoint pair" in $\ell^2(\ZZ^n)$? In this context, the term ``$(A, A^{\dagger})$ is an essentially adjoint pair" means that the following two conditions are fulfilled: $\widetilde{A}=(A^{\dagger})^{*}$ and $\widetilde{A^{\dagger}}=A^{*}$, where $\widetilde{G}$ stands for the closure and $G^{*}$ denotes the true adjoint (that is, the operator-theoretic adjoint) of an operator $G$.

In our paper we consider a more general version of this question, namely, we allow the operators $A$ (where the symbol of $A$ is in $S^{m}_{1,0}(\ZZ^n\times\TT^n)$, $m\in\RR$) and $A^{\dagger}$ to act between Sobolev spaces:
\begin{equation}\label{E:in-1}
A\colon \dom (A)\subseteq H^{s_1}(\ZZ^n)\to H^{s_2}(\ZZ^n),\quad \dom (A)=\mathcal{S}(\ZZ^n), \quad s_1,s_2\in\RR,
\end{equation}
and
\begin{equation}\label{E:in-2}
A^{\dagger}\colon \dom (A^{\dagger})\subseteq H^{-s_2}(\ZZ^n)\to H^{-s_1}(\ZZ^n),\quad \dom (A^{\dagger})=\mathcal{S}(\ZZ^n).
\end{equation}
As $A$ and $A^{\dagger}$ map $\mathcal{S}(\ZZ^n)$ into $\mathcal{S}(\ZZ^n)$ (see proposition 3.15 in~\cite{DK-20}) and $\mathcal{S}(\ZZ^n)$ is dense in $H^{s}(\ZZ^n)$ for all $s\in\RR$ (see lemma 3.16 in~\cite{DK-20}), we see that these two definitions make sense and that the operators $A$ and $A^{\dagger}$ are densely defined.

In the main result of our article (theorem~\ref{T:main-1}) we show that if $A$ is an elliptic operator of order $m$ and if $m>s_1-s_2$, then $A$ and $A^{\dagger}$ form an essentially adjoint pair. As $A$ and $A^{\dagger}$ act between two (generally different) Hilbert spaces, in definition~\ref{D-5} below we describe the term ``essentially adjoint pair" more broadly than we indicated at end of the third paragraph of this section. Here, it is worth pointing out that in definition~\ref{D-5}, one of the members of the pair acts between the anti-duals of those Hilbert spaces that serve as the ``origin" and the ``final destination" for the action of the other member of the pair. As we showed in the appendix~\ref{S:App}, the mentioned definition is applicable to the context of~(\ref{E:in-1}) and~(\ref{E:in-2}) because the anti-dual of $H^{s}(\ZZ^n)$ can be identified suitably with the space $H^{-s}(\ZZ^n)$. We mention that in the case $s_1=s_2=0$ our theorem~\ref{T:main-1} reduces to theorem 3.19 from~\cite{DK-20}, that is, we recover the result about the equality of minimal and maximal realizations of $A|_{\mathcal{S}(\ZZ^n)}$ in $\ell^2(\ZZ^n)$ (as described earlier in this section). If we assume $s_1=s_2=0$ and $A=A^{\dagger}$ (see corollary~\ref{C:main-1}), we get the essential self-adjointness of $A$ on $\mathcal{S}(\ZZ^n)$.

We note that the special case $s_1=s_2=0$ of our theorem~\ref{T:main-1} can be obtained quickly from the following two observations: (i) by theorem 4.1 in~\cite{BRK-20} our operator $A$ is unitarily equivalent to the formal adjoint of a toroidal pseudo-differential operator $B$ whose symbol $b$ is related to the symbol $a$ of $A$ as follows: $b(x,k)=\overline{a(-k,x)}$, where $x\in \mathbb{T}^n$ and $k\in\ZZ^n$ (here $\TT^n$ is the $n$-torus) and (ii) the minimal and maximal realizations in the space $L^2(\TT^n)$ of an elliptic operator $B|_{C^{\infty}(\TT^n)}$ of order $m>0$  coincide (see theorem 3.5 in~\cite{Pirh-11} for the case $n=1$ and theorem 4.8 in~\cite{Kalleji-15} for the general case $n\in\NN$).  (Although the authors of~\cite{Kalleji-15,Pirh-11}  work in the $L^p$-setting, $1<p<\infty$, in our discussion we stress the Hilbert space case.) However, the general case (see~(\ref{E:in-1}) and~(\ref{E:in-2}) above) of the problem  considered in our article is not covered in~\cite{Kalleji-15,Pirh-11}, where the minimal and maximal realizations are operators $L^2(\TT^n)\to L^2(\TT^n)$ (focusing on the Hilbert space case $p=2$) . Thus, our theorem~\ref{T:main-1} does not follow directly (via unitary equivalence) from the mentioned results of~\cite{Kalleji-15,Pirh-11}.  The fact that the operators $A$ and $A^{\dagger}$ act between different Sobolev spaces as in~(\ref{E:in-1}) and~(\ref{E:in-2}) complicates the situation, requiring a closer look at the Sobolev scale $\{H^s(\ZZ^n)\}_{s\in\RR}$ and modifications of the approach used in~\cite{DK-20,Kalleji-15,Pirh-11}.

We prove theorem~\ref{T:main-1} by employing the symbol calculus from~\cite{BRK-20, DK-20} and by reducing the problem to the question of the essential self-adjointness of an operator acting in (a single) Hilbert space $ H^{s_1}(\ZZ^n)\oplus H^{s_2}(\ZZ^n)$. We accomplish the latter task by using an abstract device from the book~\cite{Sh-pdo-book} (recalled in lemma~\ref{L:abs-pdo-book} below) and tailoring it to our specific problem (see lemma~\ref{L:abs-pdo-book-1} below). Our theorem~\ref{T:main-1} can be viewed as a discrete analogue of the results of~\cite{Bez-77}, where the author considered the essential adjointness of the pair $(A,A^{\dagger})$, with the corresponding operators acting between Sobolev (or Sobolev-like) spaces and having hypoelliptic symbols of the following types: (i) the usual (uniform) H\"ormander class, (ii) the symbol class defined in section IV.23 of the book~\cite{Sh-pdo-book} (see the paper~\cite{Rodino-paper} for a weighted version of this class), and (iii) the (uniform) H\"ormander class with the additional condition that its members are uniformly almost periodic functions with respect to the space variable.

Lastly, we remark that it is possible to formulate an analogue of theorem~\ref{T:main-1} for operators with $M^{m}_{\rho,\Lambda}(\ZZ^n\times\mathbb{T}^n)$-type symbols from~\cite{km-21}. To keep our presentation simpler, we chose to work in the setting
of $S^{m}(\ZZ^n\times\mathbb{T}^n)$-type symbols from~\cite{BRK-20, DK-20}.

Our paper consists of six sections and an appendix. In section~\ref{S:res}, after describing the basic notations and recalling the definitions of the symbol class and the corresponding pseudo-differential operator on $\ZZ^n$, discrete Sobolev spaces, and the notion of an essentially adjoint pair of operators, we state the main result of the paper (theorem~\ref{T:main-1}). Subsequently, in section~\ref{s-pl} we recall the needed elements of the symbol calculus from~\cite{BRK-20}, state and prove a couple of additional properties of elliptic operators on $\ZZ^n$, and recall an abstract device from~\cite{Sh-pdo-book} mentioned in the previous paragraph. Finally, we prove the main result (and its corollary) in sections~\ref{S:pf-1} and~\ref{pf-c-1}. In section~\ref{SS-1-21} we describe two applications of our result: well-posedness of initial-value problems for evolution equations in $\ell^2(\ZZ^n)$ and construction of an extended Hilbert scale on $\ZZ^n$ generated by an elliptic operator. In the appendix~\ref{S:App} we discuss the anti-duality between the discrete Sobolev spaces $H^{s}(\ZZ^n)$ and $H^{-s}(\ZZ^n)$, $s\in\RR$.

\section{Notations and Results}\label{S:res}
\subsection{Basic notations}\label{SS:s-2-1}
In this article, the symbols $\ZZ$, $\NN$, and $\NN_0$ refer to the sets of integers, positive integers, and non-negative integers respectively. For $n\in\NN$, we denote the $n$-dimensional integer lattice by $\ZZ^n$. By an $n$-dimensional multiindex $\alpha$ we mean an element of $\NN_0^{n}$, that is,  $\alpha=(\alpha_1, \alpha_2, \dots, \alpha_n)$ with $\alpha_j\in \NN_0$.  We define $|\alpha|:=\alpha_1+ \alpha_2+ \dots+ \alpha_n$, and $\alpha!:=\alpha_1\alpha_2\dots\alpha_n$. For $k\in \ZZ^n$ and $\alpha\in\NN_0^{n}$, we define
\begin{equation*}
k^{\alpha}:=k_1^{\alpha_1}k_2^{\alpha_2}\dots k_n^{\alpha_n}.
\end{equation*}
and
\begin{equation*}
|k|:=\sqrt{k_1^2+k_2^2+\dots+k_n^2}.
\end{equation*}

By $\{e_j\}_{j=1}^{n}$ we mean a collection of elements $e_j\in \NN_{0}^n$, where
\begin{equation*}
e_j:=(0,0,\dots, 1,0, \dots, 0),
\end{equation*}
with $1$ occupying the $j$-th slot and $0$ occupying the remaining slots.

For a function $u(k_1,k_2,\dots,k_n)$ of the input variable $k=(k_1,k_2,\dots, k_n)\in \ZZ^n$ we define the first partial difference operator $\del_{k_j}$ as
\begin{equation*}
\del_{k_j}u(k):=u(k+e_j)-u(k),
\end{equation*}
where $k+e_j$ is the usual addition of the $n$-tuplets $k$ and $e_j$.
For a multiindex $\alpha\in \NN_0^{n}$ we define
\begin{equation*}
\del^{\alpha}_k:=\del^{\alpha_1}_{k_1}\del^{\alpha_2}_{k_2}\dots \del^{\alpha_n}_{k_n}.
\end{equation*}
We will also need basic differential operators on the $n$-dimensional torus $\mathbb{T}^n:=\RR^n/\ZZ^n$. For $x\in \mathbb{T}^n$ and $\alpha\in \NN^n$ we define
\begin{equation*}
D_{x_j}:=\frac{1}{2\pi i}\frac{\p}{\p{x_j}},\qquad D^{\alpha}_x:=D^{\alpha_1}_{x_1}D^{\alpha_2}_{x_2}\dots D^{\alpha_n}_{x_n},
\end{equation*}
where $i$ is the imaginary unit. Additionally, for $l\in \NN_0$ we define,
\begin{equation*}
D^{(l)}_{x_j}:=\prod_{r=0}^{l-1}\left(\frac{1}{2\pi i}\frac{\p}{\p{x_j}}-r\right),\qquad D^{(0)}_{x_j}:=1,
\end{equation*}
where ``1" refers to the identity operator. For $\alpha\in \NN_0^{n}$, we define
\begin{equation*}
D^{(\alpha)}_{x}:=D^{(\alpha_1)}_{x_1}D^{(\alpha_2)}_{x_2}\dots D^{(\alpha_n)}_{x_n}.
\end{equation*}

\subsection{Symbol classes}\label{SS:s-2-2} We now recall the definition of the symbol class $\sml$ introduced by the authors of~\cite{BRK-20}.
\begin{definition}\label{D-1} For $m\in\RR$, the notation $S^{m}(\ZZ^n\times \TT^n)$ indicates the set of functions $a\colon \ZZ^n\times \TT^n\to\mathbb{C}$ satisfying the following properties:
\begin{enumerate}
\item [(i)] for all $k\in \ZZ^n$, we have $a(k,\cdot)\in C^{\infty}(\TT^n)$;
\item [(ii)] for all $\alpha,\,\beta\in \NN_0^{n}$, there exists a constant $C_{\alpha,\beta}>0$ such that
\begin{equation*}
 |D_{x}^{(\beta)}\Delta^{\alpha}_{k}a(k,x)|\leq C_{\alpha,\beta}(1+|k|)^{m-|\alpha|},
\end{equation*}
for all $(k,x)\in \ZZ^n\times \TT^n$.
\end{enumerate}
\end{definition}

\begin{remark} The class $S^{m}(\ZZ^n\times \TT^n)$ is not contained in the symbol classes introduced in definition 2.1  and definition 4.1 of~\cite{Rab-4}. The symbol class of definition 2.1 (definition 4.1) of~\cite{Rab-4} is designed so that the corresponding operator is bounded in unweighted (suitably weighted) $\ell^p$-spaces on $\ZZ^n$, $1\leq p\leq\infty$. Lastly, we remark that our main result (see theorem~\ref{T:main-1} below) does not follow from the boundedness results of~\cite{Rab-4}.
\end{remark}

We also recall the definition of an elliptic symbol from~\cite{BRK-20}.

\begin{definition}\label{D-2}  For $m\in\RR$, the elliptic symbol class $ES^{m}(\ZZ^n\times \TT^n)$ refers to the set of functions $a\in S^{m}(\ZZ^n\times \TT^n)$  satisfying the following property: there exist constants $C>0$ and $R>0$ such that
\begin{equation*}
|a(k,x)|\geq C(1+|k|)^m,
\end{equation*}
for all $x\in\TT^n$ and all $k\in \ZZ^n$ such that $|k|>R$.
\end{definition}

\subsection{Schwartz space}\label{SS:s-2-3} Before discussing pseudo-differential operators on $\ZZ^n$, we recall the definition of the Schwartz space $\mathcal{S}(\ZZ^n)$ from~\cite{BRK-20} and~\cite{DK-20}: this space consists of the functions $u\colon\ZZ^n\to\CC$ such that for all $\alpha,\beta\in\NN_0^n$ we have
\begin{equation*}
\displaystyle\sup_{k\in\ZZ^n}|k^{\alpha}(\del_{k}^{\beta}u)(k)|<\infty.
\end{equation*}
The symbol $S'(\ZZ^n)$ indicates the space of tempered distributions, that is, continuous linear functionals on $\mathcal{S}(\ZZ^n)$.

\subsection{Discrete Fourier transform}\label{SS:s-2-4}
For $1\leq p<\infty$ we define $\ell^p(\ZZ^n)$ as the space of functions $u\colon \ZZ\to \CC$ such that $\|u\|_{p}<\infty$, where
\begin{equation*}
\|u\|_{p}^{p}:=\sum_{k\in\ZZ^n}|u(k)|^p.
\end{equation*}
In particular for $p=2$ we get a Hilbert space $\ell^2(\ZZ^n)$ with the inner product
\begin{equation}\label{E:inner-l-2}
  (u,v):=\sum_{k\in\ZZ^n}u(k)\overline{v(k)}.
\end{equation}
To simplify the notation we will denote the corresponding norm in $\ell^2(\ZZ^n)$ by $\|\cdot\|$.

For $u\in \ell^1(\ZZ^n)$, its discrete Fourier transform $\widehat{u}(x)$ is a function of $x\in\TT^n$ defined as
\begin{equation*}
\widehat{u}(x):=\sum_{k\in\ZZ^n}e^{-2\pi i k\cdot x}u(k),
\end{equation*}
where $k\cdot x:=k_1x_1+k_2x_2+\dots +k_nx_n$.
It turns out that the discrete Fourier transform can be extended to $\ell^2(\ZZ^n)$, and by normalizing the Haar measure on $\ZZ^n$ and $\TT^n$, the Plancherel formula takes the following form:
\begin{equation*}
\sum_{k\in\ZZ^n}|u(k)|^2=\int_{\TT^n}|\widehat{u}(x)|^2\,dx.
\end{equation*}
We also recall the following inversion formula:
\begin{equation}\label{E:inv-ft}
u(k)=\int_{\TT^n}e^{2\pi i k\cdot x}\widehat{u}(x)\,dx,\qquad k\in\ZZ^n.
\end{equation}

\subsection{Pseudo-differential operator}\label{SS:s-2-5} For a symbol $a\in S^m(\ZZ^n\times \TT^n)$ we define the pseudo-differential operator $T_{a}$ as follows:
\begin{equation}\label{E:op-a}
  (T_{a}u)(k):= \int_{\TT^n} e^{2\pi i k\cdot x}a(k,x)\widehat{u}(x)\,dx,\quad u\in \mathcal{S}(\ZZ^n).
\end{equation}

In proposition 3.15 of~\cite{DK-20}, the authors showed that the operator $T_{a}$ maps $\mathcal{S}(\ZZ^n)$ into $\mathcal{S}(\ZZ^n)$, and by a small modification of this argument, one can show that $T_{a}\colon \mathcal{S}(\ZZ^n)\to \mathcal{S}(\ZZ^n)$ is continuous.

For a continuous linear operator $P\colon \mathcal{S}(\ZZ^n)\to  \mathcal{S}(\ZZ^n)$, its  \emph{formal adjoint} $P^{\dagger}$ is defined using the following relation:
\begin{equation}\label{E:dual-1}
  (Pu,v)=(u, P^{\dagger} v),
\end{equation}
for all $u,\,v\in \mathcal{S}(\ZZ^n)$, where $(\cdot,\cdot)$ is as in~(\ref{E:inner-l-2}).

The operator $T_{a}\colon \mathcal{S}(\ZZ^n)\to \mathcal{S}(\ZZ^n)$ extends to a continuous linear operator
\begin{equation*}
T_{a}\colon S'(\ZZ^n)\to S'(\ZZ^n)
\end{equation*}
defined as follows:
\begin{equation*}
(T_{a}F)(\overline{u}):=F\left(\overline{T_{a}^{\dagger}u}\right), \qquad F\in \mathcal{S}'(\ZZ^n),\,\,u\in \mathcal{S}(\ZZ^n),
\end{equation*}
where $T_{a}^{\dagger}$ is the formal adjoint of $T_{a}$ and $\overline{z}$ is the conjugate of $z\in\CC$.

\subsection{Sobolev spaces}\label{SS:s-2-6} To formulate the main result we will need the discrete Sobolev spaces, as defined in~\cite{BRK-20} and~\cite{DK-20}. For $s\in\RR$ define
\begin{equation}\label{sob-mult}
\Lambda_s(k):=(1+|k|^2)^{\frac{s}{2}}, \qquad k\in\ZZ^n.
\end{equation}
Looking at the definition~\ref{D-1}, it can be checked $\Lambda_{s}\in ES^s(\ZZ^n\times \TT^n)$. 

Next, for $s\in\RR$ we define
\begin{equation}\label{E:def-sob}
  H^{s}(\ZZ^n):=\{u\in \mathcal{S}'(\ZZ^n)\colon T_{\Lambda_{s}}u\in \ell^2(\ZZ^n)\}
\end{equation}
with the norm $\|u\|_{H^s}:=\|T_{\Lambda_{s}}u\|$, where $\|\cdot\|$ is the norm corresponding to the inner product~(\ref{E:inner-l-2}) in $\ell^2(\ZZ^n)$.

We observe that~(\ref{E:def-sob}) are, in fact, polynomially weighted sequence spaces. An important property, established in lemma 3.16 of~\cite{DK-20}, is the density of the space $\mathcal{S}(\ZZ^n)$ in $H^{s}(\ZZ^n)$ for all $s\in\RR$.

We will also use the following two notations:
\begin{equation*}
H^{-\infty}(\ZZ^n):=\displaystyle\cup_{s\in\RR}H^{s}(\ZZ^n), \qquad H^{\infty}(\ZZ^n):=\displaystyle\cap_{s\in\RR}H^{s}(\ZZ^n).
\end{equation*}

As shown in lemma~\ref{L:pairing-sob} in the appendix, the inner product~(\ref{E:inner-l-2}), considered as a sesquilinear form on $\mathcal{S}(\ZZ^n)\times \mathcal{S}(\ZZ^n)$, extends to a continuous sesquilinear pairing
\begin{equation}\label{E:pair-1}
 (\cdot,\cdot)\colon H^{s}(\ZZ^n)\times  H^{-s}(\ZZ^n)\to\CC,
\end{equation}
for all $s\in\RR$. Relative to this pairing, the spaces $H^{s}(\ZZ^n)$ and $H^{-s}(\ZZ^n)$ are anti-dual to each other.


\subsection{Operators acting between Sobolev spaces}\label{SS:s-2-7}

We are now ready to describe the main operators studied in this paper. For $s_1,\,s_2, m\in\RR$ and a symbol $a\in S^{m}(\ZZ^n\times \TT^n)$ we define $A\colon \dom (A)\subseteq H^{s_1}(\ZZ^n)\to H^{s_2}(\ZZ^n)$ as follows:
\begin{equation}\label{E:op-1}
Au:=T_{a}u,\quad \dom (A)=\mathcal{S}(\ZZ^n).
\end{equation}
Next, we define $A^{\dagger}\colon \dom (A^{\dagger})\subseteq H^{-s_2}(\ZZ^n)\to H^{-s_1}(\ZZ^n)$ as follows:
\begin{equation}\label{E:op-2}
A^{\dagger}u:=(T_{a})^{\dagger}u,\quad \dom (A^{\dagger})=\mathcal{S}(\ZZ^n).
\end{equation}
The operators $A$ and $A^{\dagger}$ make sense because $T_{a}$ and $(T_{a})^{\dagger}$ map $\mathcal{S}(\ZZ^n)$ into $\mathcal{S}(\ZZ^n)$, and $\mathcal{S}(\ZZ^n)\subseteq H^{s}(\ZZ^n)$ for all $s\in\RR$. Furthermore, the operators are densely defined since $\mathcal{S}(\ZZ^n)$ is dense in $H^{s}(\ZZ^n)$ for all $s\in\RR$.

\subsection{Some abstract concepts}\label{SS:2-6} Before stating the main theorem, let us recall some abstract notions. Let $\mathscr{H}_1$ and $\mathscr{H}_2$ be Hilbert spaces with their anti-dual spaces $(\overline{\mathscr{H}_1})^{'}$ and $(\overline{\mathscr{H}_2})^{'}$ (here, by ``the anti-dual space to $\mathscr{H}_j$'' we mean the space of continuous anti-linear functionals on $\mathscr{H}_j$). Let $\langle f,u\rangle_{d_j}$ denote the action of a functional $f\in(\overline{\mathscr{H}_{j}})^{'}$ on a vector $u\in \mathscr{H}_j$, $j=1,2$.
\begin{definition}\label{D-f-5}
We say that the (densely defined) operators $Y\colon \mathscr{H}_1\to \mathscr{H}_2$ and $Z\colon (\overline{\mathscr{H}_2})^{'}\to (\overline{\mathscr{H}_1})^{'}$ constitute a formally adjoint pair if
\begin{equation*}
  \langle f,Yu\rangle_{d_2}=\langle Zf,u\rangle_{d_1}
\end{equation*}
for all $u\in\dom(Y)\subseteq \mathscr{H}_1$ and $f\in \dom(Z)\subseteq (\overline{\mathscr{H}_2})^{'}$.
\end{definition}
\begin{remark}\label{R:closable-rem}
From the definition it follows that the operators $Y$ and $Z$ are closable.
\end{remark}

In what follows, the notation
\begin{equation}\label{E:def-true-adj}
G^{'}\colon (\overline{\mathscr{H}_2})^{'}\to (\overline{\mathscr{H}_1})^{'}
\end{equation}
refers to the true adjoint (that is, the adjoint in the operator-theoretic sense) of an operator $G\colon \mathscr{H}_1\to \mathscr{H}_2$.

Additionally, $\widetilde {S}$ stands for the closure of a (closable) operator $S$.

\begin{definition}\label{D-5}
We say that (densely defined) operators
\begin{equation}\label{E:pair-operator}
Y\colon \mathscr{H}_1\to \mathscr{H}_2, \qquad Z\colon (\overline{\mathscr{H}_2})^{'}\to (\overline{\mathscr{H}_1})^{'}
\end{equation}
form an adjoint pair if $Y=Z^{'}$ and $Z=Y^{'}$. Similarly, we say that (densely defined) operators $Y$ and $Z$ in~(\ref{E:pair-operator}) form an essentially adjoint pair if $\widetilde{Y}=Z^{'}$ and $\widetilde{Z}=Y^{'}$.
\end{definition}

\subsection{Statement of the result}\label{SS:2-7} In this section, we state the main result of our article.

\begin{theorem}\label{T:main-1} Let $s_1,\,s_2\in \RR$. Assume that $a\in ES^m(\ZZ^n\times \TT^n)$, where $m\in \RR$ satisfies $m>s_1-s_2$. Let $A$ and $A^{\dagger}$ be as in~(\ref{E:op-1}) and~(\ref{E:op-2}) respectively. Then, the operators $A$ and $A^{\dagger}$ form an essentially adjoint pair.
\end{theorem}
\begin{remark} The special case $s_1=s_2=0$ of theorem~\ref{T:main-1} can be obtained quickly by using unitary equivalence of $A$ with an elliptic pseudo-differential operator of order $m>0$ on the torus $\TT^n$ and appealing to theorem 3.5 of~\cite{Pirh-11} for $n=1$ and theorem 4.8 of~\cite{Kalleji-15} for $n\in\NN$. However, the general case of theorem~\ref{T:main-1} does not follow directly via unitary equivalence from the mentioned results of~\cite{Kalleji-15,Pirh-11}. For more details, see the fourth-to-last paragraph of section~\ref{S:intro} in our article.
\end{remark}

As a consequence of theorem~\ref{T:main-1}, we get the following corollary:
\begin{cor}\label{C:main-1} Assume that $a\in ES^m(\ZZ^n\times \TT^n)$ with $m>0$. Let $A$ and $A^{\dagger}$ be as in~(\ref{E:op-1}) and~(\ref{E:op-2}) respectively, in both definitions with $s_1=s_2=0$. Assume that $A$ is formally self-adjoint, that is, assume that $A=A^{\dagger}$. Then $A$, as an operator in $\ell^2(\ZZ^n)$, is essentially self-adjoint on $\mathcal{S}(\ZZ^n)$.
\end{cor}

\section{Preliminary Lemmas}\label{s-pl}

We start this section by recalling the symbol calculus for $S^{m}(\ZZ^n\times \TT^n)$ developed by the authors of~\cite{BRK-20}.

\subsection{Symbol calculus for $S^{m}(\ZZ^n\times \TT^n)$.}

The following properties of the class $S^{m}(\ZZ^n\times \TT^n)$ follow easily from definition~\ref{D-1}:

\begin{lemma}\label{L-1} For $m_1,\,m_2\in\RR$, we have
\begin{enumerate}
\item [(i)] If $m_1\leq m_2$, then  $S^{m_1}(\ZZ^n\times \TT^n)\subseteq S^{m_2}(\ZZ^n\times \TT^n)$;
\item[(ii)] If $a\in S^{m_1}(\ZZ^n\times \TT^n)$  and $b\in S^{m_2}(\ZZ^n\times \TT^n)$, then $ab\in S^{m_1+m_2}(\ZZ^n\times \TT^n)$;
\item[(iii)] If $a\in  S^{m_1}(\ZZ^n\times \TT^n)$ then $D_{x}^{(\beta)}\Delta^{\alpha}_{k}a(k,x)\in S^{m_1-|\alpha|}(\ZZ^n\times \TT^n)$, for all $\alpha,\,\beta\in \NN_0^{n}$.
\end{enumerate}
\end{lemma}

For the proof of the asymptotic sum formula given below, see lemma 3.4 in~\cite{BRK-20}.

\begin{lemma}\label{L-2} Let $\{m_j\}_{j\in \NN}$ be a sequence of real numbers satisfying the following two properties: $m_j>m_{j+1}$ for all $j$  and $\displaystyle\lim_{j\to\infty}m_j=-\infty$. Additionally, let $\{a_j\}_{j\in \NN}$  be a sequence of symbols such that $a_j\in  S^{m_j}(\ZZ^n\times \TT^n)$. Then, there exists a symbol $a\in  S^{m_1}(\ZZ^n\times \TT^n)$ such that
\begin{equation*}
  a(k,x)\sim \displaystyle\sum_{j}a_j(k,x),
\end{equation*}
where the asymptotic relation is understood in the following sense:
\begin{equation*}
\left[a(k,x)-\displaystyle \sum_{j<N}a_j(k,x)\right]\in S^{m_{N}}(\ZZ^n\times \TT^n),
\end{equation*}
for all $N\in\NN$.
\end{lemma}

We now recall the composition rule and the adjoint rule from theorems 3.1 and~3.2 in~\cite{BRK-20}, noting that the ordering of the difference (or differential) operators in the composition rule is different from that of the analogous formula for the $\RR^n$-setting.

\begin{prop}\label{L-3} Assume that $a\in S^{m_1}(\ZZ^n\times \TT^n)$ and $b\in S^{m_2}(\ZZ^n\times \TT^n)$, with $m_1,\,m_2\in\RR$. Let $T_a=\textrm{Op}[a]$ and $T_b=\textrm{Op}[b]$ be as in~(\ref{E:op-a}). Then, the following properties hold:
\begin{enumerate}
  \item [(i)] $T_aT_b=\textrm{Op}[c]$ with $c\in S^{m_1+m_2}(\ZZ^n\times \TT^n)$, and
\begin{equation}\label{E:ab-exp}
c(k,x)\sim \displaystyle\sum_{\alpha} \frac{1}{\alpha !}D_{x}^{(\alpha)}a(k,x)\Delta_{k}^{\alpha}b(k,x),
\end{equation}
where the asymptotic relation is interpreted as
\begin{equation*}
\left[c(k,x)-\displaystyle \sum_{|\alpha|<N}\frac{1}{\alpha !}D_{x}^{(\alpha)}a(k,x)\Delta_{k}^{\alpha}b(k,x)\right]\in S^{m_1+m_2-N}(\ZZ^n\times \TT^n),
\end{equation*}
for all $N\in \NN$.
  \item [(ii)] $(T_a)^{\dagger}=\textrm{Op}[q]$ with $q\in S^{m_1}(\ZZ^n\times \TT^n)$, and
\begin{equation}\label{E-asym-adj}
q(k,x)\sim \displaystyle\sum_{\alpha} \frac{1}{\alpha !}\Delta_{k}^{\alpha}D_{x}^{(\alpha)}\overline{a(k,x)},
\end{equation}
where $\overline{z}$ denotes the conjugate of $z\in\mathbb{C}$, and the asymptotic relation is interpreted as
\begin{equation*}
\left[q(k,x)-\displaystyle \sum_{|\alpha|<N}\frac{1}{\alpha !}\Delta_{k}^{\alpha}D_{x}^{(\alpha)}\overline{a(k,x)}\right]\in S^{m_1-N}(\ZZ^n\times \TT^n),
\end{equation*}
for all $N\in \NN$.
\end{enumerate}
\end{prop}

For the next proposition, which summarizes the continuity property, see corollary 5.6 in~\cite{BRK-20}.

\begin{prop}\label{L-6} Assume that $a\in S^{m}(\ZZ^n\times \TT^n)$ with $m\in\mathbb{R}$. Let $T_a=\textrm{Op}[a]$ be as in~(\ref{E:op-a}). Then, $T_{a}\colon \mathcal{S}(\ZZ^n)\to \mathcal{S}(\ZZ^n)$ extends to a continuous linear operator $T_{a}\colon H^{s}(\ZZ^n)\to H^{s-m}(\ZZ^n)$ for all $s\in\RR$.
\end{prop}

Before going further we define
\begin{equation*}
S^{-\infty}(\ZZ^n\times \TT^n):=\displaystyle\cap_{t\in\RR}S^{t}(\ZZ^n\times \TT^n).
\end{equation*}

The statement of the next proposition, which describes the parametrix property of an elliptic operator (that is, an operator whose symbol belongs to $ES^{m}(\ZZ^n\times \TT^n)$), is contained in theorem 3.6 in~\cite{BRK-20}.
\begin{prop}\label{L-5} Assume that $p\in ES^{m}(\ZZ^n\times \TT^n)$ with $m\in\mathbb{R}$. Let $T_p=\textrm{Op}[p]$  be as in~(\ref{E:op-a}). Then, there exists an operator $Q=\textrm{Op}[q]$ with $q\in S^{-m}(\ZZ^n\times \TT^n)$, such that
\begin{equation*}
  QP=I+V_1\,\quad PQ=I+V_2,
\end{equation*}
where $I$ is the identity operator and $V_j=\textrm{Op}[v_j]$ with $v_j\in S^{-\infty}(\ZZ^n\times \TT^n)$, $j=1,2$.
\end{prop}

As a consequence of propositions~\ref{L-6} and~\ref{L-5}, one easily gets the following elliptic regularity property:

\begin{prop}\label{L-7} Assume that $a\in ES^{m}(\ZZ^n\times \TT^n)$ with  $m\in\mathbb{R}$.  Let $T_{a}$ be as in~(\ref{E:op-a}). Additionally, assume that $u\in H^{-\infty}(\ZZ^n)$ and $T_{a}u\in \sob^{s}(\ZZ^n)$, for some $s\in \RR$. Then, $u\in \sob^{s+m}(\ZZ^n)$.
\end{prop}

\subsection{Further properties of the class $ES^{m}(\ZZ^n\times\TT^n)$.}

In this section we describe two additional features of the class $ES^{m}(\ZZ^n\times\TT^n)$.

\begin{lemma}\label{L-1-hyp} Assume that $a \in ES^{m_1}(\ZZ^n\times\TT^n)$, $b\in ES^{m_2}(\ZZ^n\times\TT^n)$, and $c\in S^{m_3}(\ZZ^n\times\TT^n)$, with $m_j\in\RR$, $j=1,2,3$.
\begin{enumerate}
  \item [(i)]  Then  $ab\in ES^{m_1+m_2}(\ZZ^n\times\TT^n)$.
  \item [(ii)] If $m_3<m_1$, then $(a+c)\in ES^{m_1}(\ZZ^n\times\TT^n)$.
\end{enumerate}
\end{lemma}
\begin{proof} Property (i) is easily verified by using the equality $|ab|=|a||b|$ and definition~\ref{D-2}. To verify the property (ii), we use the hypothesis $c\in S^{m_3}(\ZZ^n\times\TT^n)$, which tells us
\begin{equation*}
|c(k,x)|\leq C_1 (1+|k|)^{m_3},
\end{equation*}
where $C_1$ is a constant independent of $x$. We also use the absolute value inequality
\begin{eqnarray*}
|a(k,x)+c(k,x)|\geq |a(k,x)|-|c(k,x)|
\end{eqnarray*}
and definition~\ref{D-2}, which tells us
\begin{eqnarray*}
|a(k,x)|\geq C_2(1+|k|)^{m_1},\qquad |k|> R,
\end{eqnarray*}
where $C_2$ and $R$ are constants. Combining the mentioned three inequalities together with the assumption $m_3<m_1$, we obtain $(a+c)\in ES^{m_1}(\ZZ^n\times\TT^n)$.
\end{proof}

We now record two properties of the operators whose symbols belong to the class $ES^{m}(\ZZ^n\times\TT^n)$.

\begin{lemma}\label{L-4-hyp} Assume that $T_{a}=\textrm{Op}[a]$ and $T_{b}=\textrm{Op}[b]$ with $a \in ES^{m_1}(\ZZ^n\times\TT^n)$, $b\in ES^{m_2}(\ZZ^n\times\TT^n)$, with $m_j\in\RR$, $j=1,2$. Then
\begin{enumerate}
  \item [(i)]  $T_aT_b=\textrm{Op}[c]$ with $c\in ES^{m_1+m_2}(\ZZ^n\times\TT^n)$;
  \item [(ii)] $(T_a)^{\dagger}=\textrm{Op}[c]$ with $c\in ES^{m_1}(\ZZ^n\times\TT^n)$.
\end{enumerate}
\end{lemma}
\begin{proof} Let us verify the property (i). From~(\ref{E:ab-exp}) we get
\begin{equation}\label{E:asympt-bracket}
c\sim ab + \sum_{|\alpha|\geq 1} \frac{1}{\alpha !}D_{x}^{(\alpha)}a(k,x)\Delta_{k}^{\alpha}b(k,x).
\end{equation}
By lemma~\ref{L-1-hyp}(i) we have $ab\in ES^{m_1+m_2}(\ZZ^n\times\TT^n)$. Turning to the expression after the summation symbol, using the parts (iii) and (ii) of lemma~\ref{L-1},  we get
$[D_{x}^{(\alpha)}a\Delta_{k}^{\alpha}b]\in S^{m_1+m_2-|\alpha|}(\ZZ^n\times\TT^n)$. Therefore, referring to lemma~\ref{L-2}, we see that the term after $+$ on the right hand side of~(\ref{E:asympt-bracket}) leads to a symbol belonging to  $S^{m_1+m_2-1}(\ZZ^n\times\TT^n)$. Since $ab\in ES^{m_1+m_2}(\ZZ^n\times\TT^n)$ we can use part (ii) of lemma~\ref{L-1-hyp} to infer $c\in ES^{m_1+m_2}(\ZZ^n\times\TT^n)$. The property (ii) can be proved in the same way after starting from the expansion~(\ref{E-asym-adj}).
\end{proof}

\subsection{Two Abstract Results}\label{S-5-1}
We recall that if $\mathscr{H}_1$ and $\mathscr{H}_2$ are Hilbert spaces with inner products $(\cdot,\cdot)_{\mathscr{H}_j}$, $j=1,2$, then $\mathscr{H}_1\oplus\mathscr{H}_2$ is a Hilbert space whose inner product is defined as follows: For
\[
u=\left(
  \begin{array}{c}
    u_1 \\
    u_2 \\
  \end{array}
\right)\in \mathscr{H}_1\oplus\mathscr{H}_2,\quad
v=\left(
  \begin{array}{c}
    v_1 \\
    v_2 \\
  \end{array}
\right)\in \mathscr{H}_1\oplus\mathscr{H}_2
\]
we have
\begin{equation}\label{E-h-sum}
(u,v)_{\mathscr{H}_1\oplus\mathscr{H}_2}=(u_1,v_1)_{\mathscr{H}_1}+(u_2,v_2){\mathscr{H}_2}.
\end{equation}
Denote by $E_{j}\colon (\overline{\mathscr{H}_j})^{'}\to \mathscr{H}_j$, $j=1,2$, the isometric isomorphisms identifying anti-linear functionals with vectors via Riesz theorem:
\begin{equation}\label{Riesz}
\langle f,\cdot\rangle_{d_j}=(E_{j}f, \cdot)_{\mathscr{H}_j},
\end{equation}
where $\langle f,u\rangle_{d_j}$ indicates the action of the functional $f\in (\overline{\mathscr{H}_j})^{'}$ on the vector $u\in\mathscr{H}_{j}$.

For an operator $G\colon \mathscr{H}_1\to \mathscr{H}_2$, we define  $G^*\colon \mathscr{H}_2\to \mathscr{H}_1$ as

\begin{equation}\label{E:def-g-star}
G^*:=E_{1}G'E_{2}^{-1},
\end{equation}
where $G'\colon (\overline{\mathscr{H}_2})^{'}\to (\overline{\mathscr{H}_1})^{'}$ is the true adjoint of $G$ introduced in~(\ref{E:def-true-adj}).

For the following abstract result, we refer the reader to exercise IV.26.2 (and the outline of its proof) in the book~\cite{Sh-pdo-book}:

\begin{lemma}\label{L:abs-pdo-book} Let $\mathscr{H}_1$ and $\mathscr{H}_2$ be Hilbert spaces with inner products $(\cdot,\cdot)_{\mathscr{H}_j}$, $j=1,2$. Let $P\colon \mathscr{H}_1\to \mathscr{H}_2$ and $Q\colon \mathscr{H}_2\to \mathscr{H}_1$ be densely defined operators such that
\begin{equation}\label{E:sym-hyp}
  (Pu,v)_{\mathscr{H}_2}=(u, Qv)_{\mathscr{H}_1},
\end{equation}
for all $u\in\dom(P)\subseteq \mathscr{H}_1$ and $v\in \dom(Q)\subseteq \mathscr{H}_2$. Let $\mathcal{T}\colon \mathscr{H}_1\oplus\mathscr{H}_2\to \mathscr{H}_1\oplus\mathscr{H}_2$ be an operator given by
\begin{equation}\label{E:matrix-op}
\mathcal{T}=\left(
              \begin{array}{cc}
                0 & Q \\
                P & 0 \\
              \end{array}
            \right), \qquad \mathcal{T}\left(\begin{array}{c}
                u_1  \\
                u_2 \\
              \end{array}\right)=\left(\begin{array}{c}
                Qu_2  \\
                Pu_1 \\
              \end{array}\right),
\end{equation}
with the domain $\dom \mathcal{T}=\dom (P)\oplus \dom (Q)$. (Here, $\mathscr{H}_1\oplus\mathscr{H}_2$ is a Hilbert space with the inner product~(\ref{E-h-sum}).) Then the following are equivalent:
\begin{enumerate}
  \item [(i)] The operators $P$ and $Q$ satisfy $\widetilde{P}=Q^*$ and $\widetilde{Q}=P^*$. (Here, $\widetilde{S}$ refers to the closure of an operator $S$ and $S^*$ is as explained above the statement of this lemma.)
  \item [(ii)] The operator $\mathcal{T}$ is essentially self-adjoint on $\dom (P)\oplus \dom (Q)$.
\end{enumerate}
\end{lemma}

 Before stating the next lemma, which will help us apply the previous result to our problem, we remind the reader that we use the terms ``formally adjoint pair" and ``essentially adjoint pair" in the same sense as in definitions~\ref{D-f-5} and ~\ref{D-5} above.

\begin{lemma}\label{L:abs-pdo-book-1} Assume that  $Y\colon \mathscr{H}_1\to \mathscr{H}_2$ and $Z\colon (\overline{\mathscr{H}_2})^{'}\to (\overline{\mathscr{H}_1})^{'}$ constitute a formally adjoint pair. Let $\mathcal{T}$ be as in~(\ref{E:matrix-op}) with $P=Y$ and $Q=E_{1}ZE_{2}^{-1}$, with the isometric isomorphisms $E_{j}\colon (\overline{\mathscr{H}_j})^{'}\to \mathscr{H}_j$ as in Riesz theorem.
Then the following are equivalent:
\begin{enumerate}
  \item [(i)] the operators $Y$ and $Z$ form an essentially adjoint pair;
  \item [(ii)] the operator $\mathcal{T}$ is essentially self-adjoint on $\dom (Y)\oplus E_2\dom (Z)$.
\end{enumerate}
\end{lemma}
\begin{proof} Using the property~(\ref{Riesz}) together with the assumption that $Y$ and $Z$ constitute a formally adjoint pair, we obtain
\begin{align}\label{E:symm-math-T}
&(v,Pu)_{\mathscr{H}_2}=(v,Yu)_{\mathscr{H}_2}=\langle E_2^{-1}v, Yu\rangle_{d_2} =\langle ZE_2^{-1}v, u\rangle_{d_1}\nonumber\\
&=(E_1ZE_2^{-1}v, u)_{\mathscr{H}_1}=(Qv, u)_{\mathscr{H}_1},
\end{align}
for all $u\in\dom(P)\subseteq \mathscr{H}_1$ and $v\in \dom(Q)\subseteq \mathscr{H}_2$.

This shows that~(\ref{E:sym-hyp}) is satisfied. Therefore,
$P=Y$ and $Q=E_1ZE_2^{-1}$ satisfy the hypotheses of lemma~\ref{L:abs-pdo-book}.

By definition~\ref{D-5}, the operators $Y$ and $Z$ form an essentially adjoint pair if and only if $\widetilde{Y}=Z'$ and $\widetilde{Z}=Y'$. This is equivalent to the fulfillment of the following two equalities: $\widetilde{P}=Q^*$ and $\widetilde{Q}=P^*$. This, in turn, is equivalent (in view of lemma~\ref{L:abs-pdo-book}) to the essential self-adjointness of $\mathcal{T}$ on $\dom (Y)\oplus E_2\dom (Z)$.
\end{proof}

\section{Proof of Theorem~\ref{T:main-1}}\label{S:pf-1}

To simplify the notations, in this section we will write $S^m$, $ES^{m}$, $\sob^{s}$, and $\ell^2$ instead of $S^m(\ZZ^n\times\TT^n)$, $ES^{m}(\ZZ^n\times\TT^n)$, $H^{s}(\ZZ^n)$, and $\ell^2(\ZZ^n)$. Additionally, $\widetilde{F}$ denotes the closure of an operator $F$, and $F'$ indicates the (operator) adjoint of $F$.

Let $a\in ES^{m}$ with $m\in\RR$, let $A$ and $A^{\dagger}$ be as in~(\ref{E:op-1}) and~(\ref{E:op-2}), and let $s_1,\,s_2\in\RR$.

Using the isomorphisms $H^{-s_j}(\ZZ^n)\cong (\overline{H^{s_j}(\ZZ^n)})^{'}$, $j=1,2$, established lemma~\ref{L:pairing-sob}, we see that
\begin{equation}\label{E:f-a-p-1}
  \langle f,Au\rangle_{d_2}=(f, Au)=(A^{\dagger}f, u)=\langle A^{\dagger}f,u\rangle_{d_1},
\end{equation}
for all $f\in \mathcal{S}(\ZZ^n)$ and $u\in \mathcal{S}(\ZZ^n)$, where in the second equality we the used the definition~(\ref{E:dual-1}). Here, $\langle\cdot,\cdot\rangle_{d_j}$ indicate the action of a functional belonging to $(\overline{H^{s_j}(\ZZ^n)})^{'}$ on a vector belonging to $H^{s_j}(\ZZ^n)$.

From~(\ref{E:f-a-p-1}) we see that $A$ and $A^{\dagger}$ constitute a formally adjoint pair in the sense of definition~\ref{D-f-5} above. For the remainder of the proof we will use the hypothesis $m>s_1-s_2$.

Recalling the definition of $\sob^{s}$, we see that operator $T_{\Lambda_{t}}\colon \sob^{s}\to \sob^{s-t}$, $t,\,s\in\RR$, is an isometric isomorphism with the inverse $(T_{\Lambda_{t}})^{-1}=T_{\Lambda_{-t}}$.  In particular, consider the isometric isomorphisms $E_j:=T_{\Lambda_{-2s_j}}\colon \sob^{-s_j}\to \sob^{s_j}$, $j=1,2$, with the inverses $E_{j}^{-1}=T_{\Lambda_{2s_j}}$.

As $A$ and $A^{\dagger}$ constitute a formally adjoint pair, it follows that the hypotheses of lemma~\ref{L:abs-pdo-book-1} are satisfied for $\mathscr{H}_{j}=\sob^{s_j}$, $Y=A$, and $Z=A^{\dagger}$, where $A$ and $A^{\dagger}$ are as in~(\ref{E:op-1}) and~(\ref{E:op-2}). Therefore, showing that $A$ and $A^{\dagger}$ form an essentially adjoint pair is equivalent to demonstrating that the operator $\mathcal{T}$ in~(\ref{E:matrix-op}), with $P=A$ and $Q=E_{1}A^{\dagger}E_{2}^{-1}$  is essentially self-adjoint on $\mathcal{S}(\ZZ^n)\oplus T_{\Lambda_{-2s_2}}\mathcal{S}(\ZZ^n)$.

We now transfer the problem to the context of the (more convenient) space $\ell^2\oplus \ell^2$. To this end, we define
\begin{equation}\label{E:b-t-b-0}
\mathcal{J}_{+}:=\left(
                                            \begin{array}{cc}
                                             T_{\Lambda_{s_1}}  & 0 \\
                                              0 & T_{\Lambda_{s_2}} \\
                                            \end{array}
                                          \right),\qquad \mathcal{J}_{-}:=\left(\begin{array}{cc}
                                             T_{\Lambda_{-s_1}}  & 0 \\
                                              0 & T_{\Lambda_{-s_2}} \\
                                            \end{array}
                                          \right),
\end{equation}
and observe that
\[
\mathcal{J}_{+}\colon \sob^{s_1}\oplus \sob^{s_2}\to \ell^2\oplus \ell^2
\]
is an isometric isomorphism.

As in~(\ref{E:def-true-adj}),
\[
(\mathcal{J}_{+})'\colon (\overline{\ell^2\oplus \ell^2})'\to (\overline{\sob^{s_1}\oplus \sob^{s_2}})'
\]
denotes the true adjoint of $\mathcal{J}_{+}$.

Using lemma~\ref{L:pairing-sob}, we identify $(\overline{\ell^2\oplus \ell^2})'$ with $\ell^2\oplus \ell^2$ and $(\overline{\sob^{s_1}\oplus \sob^{s_2}})'$ with $\sob^{-s_1}\oplus \sob^{-s_2}$. Having done this, we define (as in~(\ref{E:def-g-star}))
\[
(\mathcal{J}_{+})^{*}:=\mathcal{E}(\mathcal{J}_{+})',
\]
where
\[
\mathcal{E}:=\left(
                                            \begin{array}{cc}
                                             T_{\Lambda_{-2s_1}}  & 0 \\
                                              0 & T_{\Lambda_{-2s_2}} \\
                                            \end{array}
                                          \right).
\]
With these definitions in place, a computation shows that
\[
(\mathcal{J}_{+})^{*}=(\mathcal{J}_{+})^{-1}=\mathcal{J}_{-}.
\]
Therefore, the essential self-adjointness of $\mathcal{T}$, considered as an operator in $\sob^{s_1}\oplus \sob^{s_2}$ with $\dom (\mathcal{T})=\mathcal{S}(\ZZ^n)\oplus T_{\Lambda_{-2s_2}}\mathcal{S}(\ZZ^n)$, is equivalent to the essential self-adjointness of the operator
\[
\mathcal{G}:=\mathcal{J}_{+}\mathcal{T}\mathcal{J}_{-},
\]
considered as an operator in $\ell^2\oplus \ell^2$ with $\dom(\mathcal{G})=\mathcal{S}(\ZZ^n)\oplus \mathcal{S}(\ZZ^n)$.

Remembering that $E_j=T_{\Lambda_{-2s_j}}\colon \sob^{-s_j}\to \sob^{s_j}$, $j=1,2$, and computing
\begin{align*}
&\mathcal{G}=\left(                           \begin{array}{cc}
                                             T_{\Lambda_{s_1}}  & 0 \\
                                              0 & T_{\Lambda_{s_2}} \\
                                            \end{array}
                                          \right)\left(\begin{array}{cc}
                                             0  &  E_{1}A^{\dagger}E_{2}^{-1}\\
                                              A & 0 \\
                                            \end{array}\right)\left(\begin{array}{cc}
                                             T_{\Lambda_{-s_1}}  & 0 \\
                                              0 & T_{\Lambda_{-s_2}} \\
                                            \end{array}
                                          \right)\nonumber\\
                                          &=\left(\begin{array}{cc}
                                              0 & T_{\Lambda_{-s_1}}A^{\dagger}T_{\Lambda_{s_2}} \\
                                              T_{\Lambda_{s_2}}A T_{\Lambda_{-s_1}}& 0 \\
                                            \end{array}
                                          \right),
\end{align*}
we note that $\mathcal{G}$ can be written as
\begin{equation}\label{E:b-t-b-1}
\mathcal{G}=\left(\begin{array}{cc}
                                              0 &  B^{\dagger} \\
                                              B& 0 \\
                                            \end{array}
                                          \right),
\end{equation}
where $B:=T_{\Lambda_{s_2}}A T_{\Lambda_{-s_1}}$ is elliptic of order $r:=m+s_2-s_1$, as granted by lemma~\ref{L-1-hyp}. (Furthermore, note that $r>0$ in view of the hypothesis $m>s_1-s_2$.)

Using~(\ref{E:b-t-b-1}) and the definition of $(\cdot,\cdot)_{\ell^2\oplus \ell^2}$ from~(\ref{E-h-sum}), we see that $\mathcal{G}$ is a symmetric operator with $\dom(\mathcal{G})=\mathcal{S}(\ZZ^n)\oplus \mathcal{S}(\ZZ^n)$. In other words, we have $\mathcal{G}\subset \mathcal{G}^{*}$, where $\mathcal{G}^{*}$ is the true adjoint of $\mathcal{G}$. (Here, we used lemma~\ref{L:pairing-sob} to identify $(\overline{\ell^2\oplus \ell^2})'$ with $\ell^2\oplus \ell^2$.) Since (by an abstract fact) $\mathcal{G}^{*}$ is a closed operator, we have
\begin{equation}\label{E:temp-lr-1}
\widetilde{\mathcal{G}}\subset \mathcal{G}^{*},
\end{equation}
where $\widetilde{\mathcal{G}}$ is the closure of $\mathcal{G}$.

Our goal is to establish the equality $\widetilde{\mathcal{G}}=\mathcal{G}^{*}$.

To reach this goal, we first note that proposition~\ref{L-6} (applied to the block-operators $B$ and $B^{\dagger}$ in~(\ref{E:b-t-b-1}) of order $r=m+s_2-s_1>0$) yields
\begin{equation}\label{E:temp-lr-2}
H^{r}\oplus H^{r}\subset \dom(\widetilde{\mathcal{G}}).
\end{equation}(This step uses the assumption $r>0$,  ensuring that $H^{r}\oplus H^{r}\subset\ell^2\oplus\ell^2$.)

Furthermore, looking at~(\ref{E:temp-lr-1}) and~(\ref{E:temp-lr-2}), we see that it remains to show the inclusion
\begin{equation}\label{E:temp-lr-3}
\dom(\mathcal{G}^{*})\subset H^{r}\oplus H^{r}.
\end{equation}

Using the definition of $\mathcal{G}^{*}$ and looking at~(\ref{E:b-t-b-1}), the inclusion
$v=\left(
  \begin{array}{c}
    v_1 \\
    v_2 \\
  \end{array}
\right)
\in\dom(\mathcal{G}^{*})$ means the following: $v_j\in\ell^2$, $j=1,2$, and there exists $f=\left(
  \begin{array}{c}
    f_1 \\
    f_2 \\
  \end{array}
\right)\in \ell^2\oplus \ell^2$ such that
\begin{equation}\label{E:temp-lr-4}
(v,\mathcal{G}u)_{\ell^2\oplus\ell^2}=(f,u)_{\ell^2\oplus\ell^2},
\end{equation}
for all $u=\left(
  \begin{array}{c}
    u_1 \\
    u_2 \\
  \end{array}
\right)\in \dom(\mathcal{G})=\mathcal{S}(\ZZ^n)\oplus \mathcal{S}(\ZZ^n)$.

Using~(\ref{E:b-t-b-1}) and the definition of $(\cdot,\cdot)_{\ell^2\oplus \ell^2}$ from~(\ref{E-h-sum}), we can rewrite~(\ref{E:temp-lr-4}) as
\begin{equation}\label{E:temp-lr-5}
(v_1,B^{\dagger}u_2)+(v_2,Bu_1)=(f_1,u_1)+(f_2,u_2),
\end{equation}
where $(\cdot,\cdot)$ is the inner product in $\ell^2$.

As~(\ref{E:temp-lr-5}) holds for all $u_1\in \mathcal{S}(\ZZ^n)$ and all $u_2\in \mathcal{S}(\ZZ^n)$, putting $u_1=0$ first, and, after that, $u_2=0$, we get
\[
(v_1,B^{\dagger}u_2)=(f_2,u_2),\qquad (v_2,Bu_1)=(f_1,u_1),
\]
where $(\cdot,\cdot)$ is the inner product in $\ell^2$.

Keeping in mind that $v_j\in\ell^2$ and taking into account lemma~\ref{L:pairing-sob}, the latter two equalities lead to
\[
Bv_1=f_2,\qquad B^{\dagger}v_2=f_1.
\]
As $B$ and $B^{\dagger}$ are elliptic (of order $r=m+s_2-s_1$) and as $f_j\in\ell^2$, $j=1,2$, proposition~\ref{L-7} tells us that $v_j\in H^{r}$, $j=1,2$. This shows the inclusion~(\ref{E:temp-lr-3}). $\hfill\square$

\section{Proof of Corollary~\ref{C:main-1}}\label{pf-c-1} In this corollary $A$ is an operator $\ell^2(\ZZ^n)$ with domain $\mathcal{S}(\ZZ^n)$. Using theorem~\ref{T:main-1} with $s_1=s_2=0$, we infer that $A$ and $A^{\dagger}$ form an essentially adjoint pair. After identifying $(\overline{\ell^2(\ZZ^n)})'$ with $\ell^2(\ZZ^n)$, the latter means (by definition~\ref{D-5}) that $\widetilde{A}=(A^{\dagger})^{*}$ and $\widetilde{A^{\dagger}}=A^{*}$, where $\widetilde{G}$ indicates the closure of an operator $G$ and $G^{*}$ indicates the true adjoint of $G$.  Now the hypothesis of the corollary tells us $A=A^{\dagger}$, which in combination with the preceding sentence leads to  $\widetilde{A}=A^{*}$. This means that $A$ is essentially self-adjoint on $\dom(A)=\mathcal{S}(\ZZ^n)$. $\hfill\square$

\begin{remark}\label{R:dom-m} The proof of theorem~\ref{T:main-1} in the case $s_1=s_2=0$ tells us that under the hypotheses of corollary~\ref{C:main-1} we have $\dom (\widetilde{A})=\dom (A^{*})=H^{m}(\ZZ^n)$. This property was established earlier in proposition 3.18 of~\cite{DK-20}.
\end{remark}

\section{Two Applications}\label{SS-1-21}
In this section we illustrate the importance of corollary~\ref{C:main-1} for well-posedness of initial-value problems for evolution equations in $\ell^2(\ZZ^n)$ and for constructing an extended Hilbert scale on $\ZZ^n$ corresponding to an elliptic operator.

\subsection{Evolution Equations}\label{S:S-7-1}
Let $G$ be an operator (with domain $\dom(G)$) on a Hilbert space $\mathscr{H}$. The problem
\begin{equation}\label{E:ivp-1}
\frac{du}{dt}=Gu(t),\quad t\geq 0,\qquad u(0)=u_0,\qquad u_0\in\mathscr{H},
\end{equation}
is called \emph{an abstract Cauchy problem} (ACP) associated with $G$ and initial value $u_0\in\mathscr{H}$.

We say $u\in[0,\infty)\to \mathscr{H}$ is a \emph{solution of the ACP}~(\ref{E:ivp-1}) if $u(\cdot)$ is continuously differentiable, $u(t)\in\dom(G)$ for all $t\geq 0$ and $u$ satisfies~(\ref{E:ivp-1}).

We recall the following property from proposition II.6.2 in~\cite{eng-nag}:
\begin{itemize}
  \item [(P)] Assume that $G$ is a generator of a strongly continuous (not necessarily contractive) semigroup $\{V(t)\}_{t\geq 0}$ of bounded linear operators on a Hilbert space $\mathscr{H}$. Then for every $u_0\in \dom(G)$, the function
      \[
      t\mapsto u(t):=V(t)u_0,
      \]
is the unique solution to the ACP~(\ref{E:ivp-1}).
\end{itemize}

\bigskip

We now consider $\mathscr{H}=\ell^2(\ZZ^n)$ and an operator $A:=\textrm{Op}[a]$, such that $a\in ES^{m}(\ZZ^n\times \TT^n)$, with $m>0$, where $ES^{m}(\ZZ^n\times \TT^n)$ is as in definition~\ref{D-2}. Furthermore, we assume that $A$ is lower semi-bounded:
\begin{equation}\label{E:non-neg}
(Au,u)\geq C\|u\|^2,
\end{equation}
for all $u\in \mathcal{S}(\ZZ^n)$, where $C\in\mathbb{R}$,  $(\cdot,\cdot)$ is as in~(\ref{E:inner-l-2}) and $\|\cdot\|$ is the corresponding norm in $\ell^2(\ZZ^n)$.

It is well known that~(\ref{E:non-neg}) implies that $A=A^{\dagger}$. Thus, by corollary~\ref{C:main-1}, the closure $\widetilde{A}$ of $A|_{\mathcal{S}(\ZZ^n)}$ is a lower semi-bounded self-adjoint operator in $\ell^2(\ZZ^n)$ (with $\dom (\widetilde{A})=H^{m}(\ZZ^n)$, as indicated in remark~\ref{R:dom-m}. Therefore, by theorem II.3.15 from~\cite{eng-nag} the operator $-\widetilde{A}$ generates a strongly continuous (quasi-contractive) semigroup on $\ell^2(\ZZ^n)$. This, together with (P), yields the following property:

\bigskip

\noindent\emph{Under the above assumptions on $A:=\textrm{Op}[a]$, the ACP~(\ref{E:ivp-1}) with $\mathscr{H}=\ell^2(\ZZ^n)$, $G=-\widetilde{A}$, and $u_0\in \dom(\widetilde{A})=H^{m}(\ZZ^n)$ has a unique solution.}

\bigskip

\subsection{Extended Hilbert Scale on $\ZZ^n$}\label{SS:g-hilbert-a} We now use corollary~\ref{C:main-1} to construct an extended Hilbert scale on $\mathbb{Z}^n$ generated by an elliptic operator. It turns out that this (and related) scale(s) have attracted quite a bit of interest in the last two decades; see the monograph~\cite{MM-14}, the paper~\cite{MM-21}, and the references therein.

We begin with abstract terminology from section 2 of~\cite{MM-21}. Let $\mathscr{H}$ be a separable complex Hilbert space with inner product $(\cdot,\cdot)_{\mathscr{H}}$ and norm $\|\cdot\|_{\mathscr{H}}$. Let $A$ be a self-adjoint operator in $\mathscr{H}$ such that $(Au,u)_{\mathscr{H}}\geq\|u\|^2_{\mathscr{H}}$ for all $u\in\dom (A)$.

Applying spectral calculus we get the operator $A^{s}$ for each $s\in\RR$. Note that $\dom(A^s)$ is dense in $\mathscr{H}$; in particular, if $s\leq0$ we have  $\dom (A^s)=\mathscr{H}$. We define $H_{A}^{(s)}$ as the completion of $\dom(A^s)$ with respect to the inner product
\begin{equation*}
(u,v)_{s}:=(A^su,A^sv)_{\mathscr{H}}, \qquad u,v\in \dom(A^s).
\end{equation*}
It turns out that $H_{A}^{(s)}$ is a separable Hilbert space. In the sequel, we use the symbols $(\cdot,\cdot)_{s}$ and $\|\cdot\|_{s}$ to indicate the inner product and norm of $H_{A}^{(s)}$.

The family $\{H_{A}^{(s)}\colon s\in\RR\}$ is called \emph{Hilbert scale generated by $A$} or, in abbreviated form, \emph{$A$-scale}. According to section 2 of~\cite{MM-21}, for $s\geq 0$ we have $H_{A}^{(s)}=\dom (A^s)$, while for $s<0$ we have $H_{A}^{(s)}\supset \mathscr{H}$.

As indicated in section 2 of~\cite{MM-21}, for all $s_0<s_1$,  $[H_{A}^{(s_0)},H_{A}^{(s_1)}]$ is an admissible pair (in the sense of sections 1.1.1--1.1.2 of~\cite{MM-14}). In what follows, the term \emph{extended Hilbert scale generated by $A$} or
\emph{extended $A$-scale} refers to the set of all Hilbert spaces that serve as interpolation spaces with respect to pairs of the form $[H_{A}^{(s_0)},H_{A}^{(s_1)}]$, $s_0<s_1$. (Here we use the term \emph{interpolation space} in the sense of sections 1.1.1--1.1.2 of~\cite{MM-14}.)

Applying spectral calculus, for a Borel measurable function $\varphi\colon [1,\infty)\to(0,\infty)$,  we define a (positive self-adjoint) operator $\varphi(A)$ in $\mathscr{H}$. Furthermore, we define $H_{A}^{\varphi}$ as the completion of
$\dom(\varphi(A))$ with respect to the inner product
\begin{equation*}
(u,v)_{\varphi}:=(\varphi(A)u,\varphi(A)v)_{\mathscr{H}}, \qquad u,v\in \dom(\varphi(A)).
\end{equation*}
Referring to section 2 of~\cite{MM-21}, $H_{A}^{\varphi}$ is a separable Hilbert space. We use the symbols $(\cdot,\cdot)_{\varphi}$ and $\|\cdot\|_{\varphi}$ to denote  the inner product and norm of $H_{A}^{\varphi}$. Furthermore, as mentioned in section 2 of~\cite{MM-21}, we have $H_{A}^{\varphi}=\dom(\varphi(A))$ if and only if $0\notin\textrm{Spec}(\varphi(A))$.

\bigskip

We now switch to the setting of $\mathscr{H}=\ell^2(\ZZ^n)$. We begin by listing the assumptions on our operator $A:=\textrm{Op}[a]$:
\begin{enumerate}
\item [(H1)] $a\in ES^{1}(\ZZ^n\times \TT^n)$, where $ES^{1}(\ZZ^n\times \TT^n)$ is as in definition~\ref{D-2};

\item [(H2)] $(Au,u)\geq \|u\|^2$, for all $u\in \mathcal{S}(\ZZ^n)$, where $(\cdot,\cdot)$ is as in~(\ref{E:inner-l-2}) and $\|\cdot\|$ is the corresponding norm in $\ell^2(\ZZ^n)$.
\end{enumerate}

Under the hypotheses (H1)--(H2), $A|_{\mathcal{S}(\ZZ^n)}$ is a formally self-adjoint operator whose symbol $a$ belongs to the class $ES^{1}(\ZZ^n\times \TT^n)$. Thus, by corollary~\ref{C:main-1}, $A|_{\mathcal{S}(\ZZ^n)}$ is an essentially self-adjoint operator in $\ell^2(\ZZ^n)$.  Moreover, remark~\ref{R:dom-m} tells us that the domain of the self-adjoint closure of $A$ is the space $H^{1}(\ZZ^n)$. To keep our notations simpler, for the remainder of this section, we denote the self-adjoint closure of $A|_{\mathcal{S}(\ZZ^n)}$ again by $A$.

Specializing the abstract construction from the beginning of section~\ref{SS:g-hilbert-a} to $\mathscr{H}=\ell^2(\ZZ^n)$, for a Borel function
$\varphi\colon [1,\infty)\to(0,\infty)$ we define the space $H_{A}^{\varphi}(\ZZ^n)$ generated by an operator $A$ satisfying (H1)--(H2).

Following section 2.4.1 in~\cite{MM-14}, we say that a function $\varphi\colon [1,\infty)\to (0,\infty)$ is \emph{$RO$-varying at infinity} if
\begin{enumerate}
  \item [(i)] $\varphi$ is Borel measurable
  \item [(ii)] there exist numbers $a>1$ and $c\geq 1$ (depending on $\varphi$) such that
  \begin{equation}\label{E:RO-1}
    c^{-1}\leq\frac{\varphi(\lambda t)}{\varphi(t)}\leq c, \qquad\textrm{for all }t\geq 1,\,\,\lambda\in [1,a].
  \end{equation}
\end{enumerate}
In the sequel, the inclusion $\varphi\in RO$ means that a function $\varphi\colon [1,\infty)\to (0,\infty)$ is $RO$-varying at infinity.

\bigskip

We end this section with a property established in theorem 2.7 of~\cite{Mil-24} (analogue of theorem 5.1 in~\cite{MM-21} for the $\RR^n$-setting):

\bigskip

\noindent \emph{Assume that $\varphi\in RO$. Additionally, assume that $A:=\textrm{Op}[a]$ is an operator satisfying (H1)--(H2). Then, up to norm equivalence, we have
\begin{equation}\label{E:thm-7}
H^{\varphi}_{A}(\ZZ^n)=H^{\varphi}(\ZZ^n),
\end{equation}
where $H^{\varphi}(\ZZ^n)$ is as in~(\ref{E:def-sob}) with $\Lambda_s(k)$ replaced by
\[
\varphi\left((1+|k|^2)^{1/2}\right).
\]
}

\bigskip

\begin{remark} The condition $a\in ES^{1}(\ZZ^n\times \TT^n)$ in (H1) can be replaced by $a\in ES^{m}(\ZZ^n\times \TT^n)$, with $m>0$. In this case, the following variant of~(\ref{E:thm-7}) holds: $H^{\varphi}_{A}(\ZZ^n)=H^{\varphi_{m}}(\ZZ^n)$, where $\varphi_{m}(t):=\varphi(t^m)$, $t\geq 1$.
\end{remark}

\appendix
\section{Anti-duality between discrete Sobolev spaces}\label{S:App}
In this section we adapt the proof of lemma 4.4.4 in~\cite{BC-09} to discrete Sobolev spaces.

We begin with a brief review of terminology. Denote by $\overline{V}$ the \emph{conjugate} of a (complex) vector space $V$. Viewed as real vector spaces, $\overline{V}$ and $V$ are the same. The only difference is that the multiplication of a vector $v\in \overline{V}$ by a scalar $z\in\CC$ is defined as $\overline{z}v$, where $\overline{z}$ is the conjugate of $z$. The space of linear (respectively, anti-linear) functionals on $V$ will be denoted by $V^{'}$ (respectively, $(\overline{V})^{'}$).

Let $V_1$ and $V_2$ be two topological (complex) vector spaces. By a  \emph{pairing}  of $V_1$ and $V_2$ we mean a continuous sesquilinear  map $B\colon V_1\times V_2\to \CC$. (Here, ``sesquilinear" means linear in the first and anti-linear in the second slot.)  The pairing $B$ gives rise to continuous linear maps $\tau\colon V_1\to (\overline{V_2})^{'}$ and $\kappa\colon \overline{V_2}\to V_{1}^{'}$ as follows:
\begin{equation*}
\tau (u):=B(u,\cdot)\,\qquad \kappa(v):=B(\cdot, v).
\end{equation*}
The pairing $B$ is said to be \emph{perfect} if $\tau$ and $\kappa$ are (linear) isomorphisms.

\begin{lemma}\label{L:pairing-sob} Let $s\in\RR$. Then the $\ell^2$-inner product~(\ref{E:inner-l-2}), initially considered on $\mathcal{S}(\ZZ^n)\times \mathcal{S}(\ZZ^n)$, extends to a (continuous sesquilinear) pairing
\begin{equation*}
(\cdot,\cdot)\colon H^{s}(\ZZ^n)\times H^{-s}(\ZZ^n)\to \mathbb{C}.
\end{equation*}
The pairing $(\cdot,\cdot)$ is perfect and induces isometric (linear) isomorphisms $\overline{H^{-s}(\ZZ^n)}\cong(H^{s}(\ZZ^n))^{'}$ and $H^{s}(\ZZ^n)\cong(\overline{H^{-s}(\ZZ^n)})^{'}$.
\end{lemma}
\begin{proof}
Let $(\cdot,\cdot)$ be as in~(\ref{E:inner-l-2}) and let $\Lambda_{s}$ be as in~(\ref{sob-mult}). For all $u,v\in \mathcal{S}(\ZZ^n)$ we have
\begin{equation}\label{E:a-1}
|(u,v)|=|(\Lambda_{s}u,\Lambda_{-s}v)|\leq \|\Lambda_{s}u\|\|\Lambda_{-s}v\|=\|u\|_{H^{s}}\|v\|_{H^{-s}},
\end{equation}
where we used Cauchy--Schwarz inequality in $\ell^2(\ZZ^n)$. The last equality is justified because
\begin{equation*}
\|\Lambda_{s}u\|=\|T_{\Lambda_s}u\|, \qquad \|\Lambda_{-s}v\|=\|T_{\Lambda_{-s}}v\|,
\end{equation*}
where the last two formulas follow from the definition~(\ref{E:op-a}) and the inversion formula~(\ref{E:inv-ft}).

As $\mathcal{S}(\ZZ^n)$ is dense in the spaces $H^{\pm s}(\ZZ^n)$ (see lemma 3.16 in~\cite{DK-20}), from~(\ref{E:a-1}) we infer that the $\ell^2$-inner product extends to a (continuous sesquilinear) pairing $(\cdot,\cdot)\colon H^{s}(\ZZ^n)\times H^{-s}(\ZZ^n)\to \mathbb{C}$.

Additionally, from~(\ref{E:a-1}) we obtain
\begin{equation}\label{E:a-2}
\|u\|_{H^{s}}=\sup\{|(u,v)|\colon v\in \mathcal{S}(\ZZ^n) \textrm{ and } \|v\|_{H^{-s}}=1\}.
\end{equation}
Using the density of $\mathcal{S}(\ZZ^n)$ in the space $H^{-s}(\ZZ^n)$, observe that the right hand side of~(\ref{E:a-2}) is equal to the norm of the functional $\tau(u)\in(\overline{H^{-s}(\ZZ^n)})^{'}$ corresponding to $u\in H^{s}(\ZZ^n)$ via $\tau(u):=(u,\cdot)$.  Therefore, $\tau\colon H^{s}(\ZZ^n)\to  (\overline{H^{-s}(\ZZ^n)})^{'}$ is an isometry.

In the same way, using~(\ref{E:a-1}) we can show that the map $\kappa\colon \overline{H^{-s}(\ZZ^n)}\to (H^{s}(\ZZ^n))^{'}$ defined as
$\kappa(v):=(\cdot,v)$ is an isometry.

In particular, the maps $\tau$ and $\kappa$ are injective. From the injectivity of $\kappa$ we infer that the range of $\tau$ is dense in $(\overline{H^{-s}(\ZZ^n)})^{'}$. As $\tau$ is an isometry, it follows that $\tau$ is surjective. Therefore, $\tau$ is an isometric isomorphism. In the same way, we can show that $\kappa$ is an isometric isomorphism.
\end{proof}

\end{document}